\documentclass[reqno,11pt]{amsart}
\usepackage{amssymb, amsmath}
\usepackage{amsthm, amsfonts,mathrsfs}
\usepackage[T1]{fontenc}
\usepackage{times}
\usepackage{graphicx}
\usepackage{epsfig}
\usepackage{color}
\usepackage[all]{xypic}
\newtheorem{theorem}{Theorem}[section]

\newtheorem{lemma}{Lemma}[section]

\newtheorem{remark}{Remark}[section]

\numberwithin{equation}{section}

\newcommand{\mr}{\mathbb{R}}

\begin{document}
\title[Extrema of dynamic pressure in flows with underlying currents and infinite depth]
{Extrema of the dynamic pressure in an irrotational regular wave train with underlying currents and infinite depth}%
\author[Lili Fan]{Lili Fan}%
\address[Lili Fan]{College of Mathematics and Information Science,
Henan Normal University, Xinxiang 453007, China}
\email{fanlily89@126.com}
\author[Hongjun Gao$^{\dag}$]{Hongjun Gao$^{\dag}$}
\address[Hongjun Gao]{School of Mathematical Sciences, Institute of Mathematics,
 Nanjing Normal University, Nanjing 210023, China;  \ Institute of Mathematics, Jilin University, Changchun 130012, China}
\email{gaohj@njnu.edu.cn\, (Corresponding author)}
\author[Lei Mao]{Lei Mao}
\address[Lei Mao]{School of Mathematical Sciences, Institute of Mathematics,
 Nanjing Normal University, Nanjing 210023, China}


\begin{abstract}In this paper, we investigate the maximum and minimum of the dynamic pressure in a regular wave train with underlying currents and infinite depth respectively. The result is obtained using maximum principles in combination with exploiting some of the physical structures of the problem itself.
\end{abstract}

\date{}

\maketitle


\vskip 0.2cm


\section{Introduction}
\large
This paper aims to study the dynamic pressure in an irrotational regular wave train with uniform underlying currents and infinite depth respectively. The dynamic pressure, encoding the fluid motion, is one of the two components of the total pressure beneath a surface water wave and the other is a hydrostatic part that counteracts the force of gravity to make the fluid be in an equilibrium state. The study of the behaviour of the pressure has not only theoretical sense but also practical applications, such as the estimations of the force acting on maritime structures \cite{Cl1,Co5,ES,D2,K,Ol}.

The significative investigations of the structures and behavior of Stokes wave have been carried out by Constantin and his colleagues, including trajectories of particles \cite{Co2}, mean velocities \cite{Co4}, analyticity and symmetry of streamlines \cite{CoEs,CoEh}, pressure \cite{CS}, dynamic pressure \cite{Co1} and so on.  In regards to Stokes wave in fluid domain with infinite depth, Henry has investigated the horizontal velocity \cite{D}, the trajectories of particles \cite{D3} and the pressure \cite{D1}. The study of the irrotational regular wave trains with uniform underlying currents is due to Basu \cite{Ba}, where he obtained that when the speed of the current is greater than the wave speed, the nature of the flow fields is different from what is expected where no underlying current was considered \cite{Co2,Co,Co3} and the pressure field is unaffected due to the presence of underlying currents. There are also some other noteworthy work related to this area of nonlinear water waves \cite{G,Jo1,L1,L2,U,VO}.

Inspired by the work of Constantin \cite{Co1}, we are concerned in this paper about the influence of the uniform underlying currents and infinite depth on the extrema of the dynamic pressure. Due to the appearance of currents, taking the current strength $k$ being greater than the wave speed $c$ for example, we first eliminate the possibility of the relative mass flux $m\geq 0$ by employing the maximum principle and Hopf's maximum principle, then a reapplication of Hopf's maximum principle on the stream function $\psi$ enable us get that the horizontal velocity $u$ is greater than the wave speed $c$, which Basu obtained by analyzing the non-dimensional system of the stream function, an equivalent system for the governing equations, in \cite{Ba}. Next, we follow the idea presented in \cite{Co1} to obtain the extrema of dynamic pressure. As for dynamic pressure in deep water, the main difficulty comes from the failure of application of the maximum principle in the fluid domain with infinite depth. Base on the result about the total pressure constructed by Henry \cite{D1}, we eliminate the possibility that the extrema can be obtained in the interior of the domain by assuming the contrary. The investigation of the dynamic pressure along the boundary of fluid domain is relatively normal. In conclusion, we show that the presence of underlying currents and infinite depth makes no difference on the position of extrema of dynamic pressure, which attains its maximum and minimum at the wave crest and wave trough respectively.

The remainder of this paper is organized as follows.  In Section 2, we present the governing equations for the irrotational regular wave trains with uniform underlying current. In Section 3, we prove the maximum and minimum  of the dynamic pressure occur at the wave crest and the wave trough respectively, unaffected due to the presence of uniform underlying currents and infinite depth.
\section{The governing equations}
\large
The problems we consider are two-dimensional steady periodic irrotational gravity water waves over a homogeneous fluid with a underlying current strength $k$, the $x$-axis being the direction of wave propagation and the $y$-axis pointing vertically upwards. The bottom is assumed to be flat, given by $y=-d$ with $d>0$ representing the mean water depth, and be impermeable. Given $c>0$, assume this two-dimensional periodic steady waves travel at speed $c$; that is, the space-time dependence of the free surface $y=\eta(x)$, of the pressure $P=P(x,y)$, and of the velocity field $(u,v)$ has the form $(x-ct)$ and is periodic with period $L>0$. Using the map $(x-ct,y) \mapsto (x,y)$, we move to a new reference frame travelling alongside the wave with constant speed $c$ where the fluid flow is steady. Then the governing equations for the gravity waves are embodied by the following nonlinear free boundary problem \cite{Co}:
\begin{align}
(u-c)u_x+vu_y&=-\frac 1 \rho P_x,\quad\;\;\quad -d\leq y\leq \eta(x),\label{2.1}\\
(u-c)v_x+vv_y&=-\frac 1 \rho P_y-g,\!\quad -d\leq y\leq \eta(x),\label{2.2}\\
u_x+v_y&=0,\qquad\;\;\qquad\;-d\leq y\leq \eta(x),\label{2.3}\\
u_y&=v_x,\qquad \;\; \qquad -d\leq y\leq \eta(x),\label{2.4}\\
v&=0,\quad \; \quad \quad \quad \quad \textrm{on} \quad y=-d,\label{2.5}\\
v&=(u-c)\eta_x,\;\;\:\quad \textrm{on} \quad y=\eta(x),\label{2.6}\\
P&=P_{atm},\;\;\qquad\quad \textrm{on} \quad y=\eta(x),\label{2.7}
\end{align}
here $P_{atm}$ is the constant atmospheric pressure, $g$ is the (constant) acceleration of gravity and $\rho$ is the (constant) density. A regular wave train is a smooth solution to the governing equations \eqref{2.1}-\eqref{2.7} and satisfies \cite{Co}:

(i) $u$ and $v$ have a single crest and trough per period,

(ii) $\eta$ is strictly monotone between successive crests and troughs,

(iii) $\eta, u$ and $P$ are symmetric and $v$ is antisymmetric about crest line.

Without loss of generality, we may assign that the crest is located at $x=0$ and the trough at $x=L/2$. Using \eqref{2.3} we can define the stream function $\psi$ up to a
constant by
\begin{equation}\label{2.8}
  \psi_y=u-c,\quad\psi_x=-v,
\end{equation}
and we fix the constant by setting $\psi=0$ on $ y=\eta(x)$. We can integrate \eqref{2.8} to get $\psi=m$ on $y=-d$, for
\begin{equation}\label{2.9}
 m= - \int_{-d}^{\eta(x)}(u(x,y)-c)dy,
\end{equation}
where $m$ is the relative mass flux, and by writing
\begin{equation}\label{2.10}
  \psi(x,y)=m+\int_{-d}^{y}(u(x,s)-c)ds,
\end{equation}
we can see that $\psi$ is periodic (with period $L$) in the $x$ variable. The level sets of $\psi(x,y)$ are the streamlines of the fluid motion. Then the governing equations \eqref{2.1}-\eqref{2.7} are transformed to the equivalent system
\begin{align}
&\psi_{xx}+\psi_{yy}=0,\quad\;\;\quad \text{for}\; -d\leq y\leq \eta(x),\label{2.11}\\
&\psi=0,\quad\;\;\quad  \text{on}\; y=\eta(x),\label{2.12}\\
&\psi=m,\quad\;\;\quad  \text{on}\; y=-d,\label{2.13}\\
&\frac {\psi^2_x+\psi^2_y} {2g}+y+d=Q \quad  \text{on}\; y=\eta(x),\label{2.14}
\end{align}
where the constant $Q$ is the total head and it is expressed by
\begin{equation}\label{2.15}
\frac {\psi^2_x+\psi^2_y} {2g}+y+d+\frac {P-P_{atm}} {\rho g}=Q
\end{equation}
throughout $\{(x,y): -d\leq y\leq \eta(x)\}$ by using the equations \eqref{2.1} and \eqref{2.2}. Then the total pressure can be recovered from
\begin{equation}\label{2.16}
P=P_{atm}+\rho g Q-\rho g (y+d)-\rho\frac {\psi^2_x+\psi^2_y} {2}.
\end{equation}
Hence the dynamic pressure we concerned in this paper, defined as the difference between the total pressure $P(x,y)$ and the hydrostatic pressure $(P_{atm}-\rho g y)$, is given by
\begin{equation}\label{2.17}
p=P(x,y)-(P_{atm}-\rho g y)=\rho g (Q-d)-\rho\frac {\psi^2_x+\psi^2_y} {2}.
\end{equation}
In the setting of periodic waves traveling at a constant speed at the surface of water in an irrotational flow over flat bed, we have
\begin{equation}\label{3.1}
0=\int_{-d}^{y_0}\int_0^{L}(u_y-v_x)dxdy=\int_0^Lu(x,y_0)dx-\int_0^Lu(x,-d)dx,
\end{equation}
for $y_0\in[-d,\eta(L/2)]$. Hence
\begin{equation}\label{3.2}
\int_0^Lu(x,y_0)dx=\int_0^Lu(x,-d)dx,\quad y_0\in[-d,\eta(L/2)],
\end{equation}
and the uniform underlying current defined by
\begin{equation}\label{3.3}
k=\frac 1 L \int_0^Lu(x,-d)dx
\end{equation}
is invariant with $y$. Thus, we represent the underlying current $k$ with
\begin{equation}\label{3.4}
k=c+\frac 1 L \int_0^L\psi_y(x,-d)dx
\end{equation}
by \eqref{2.8}.
\section{The dynamic pressure}
\large
In this section we will prove the main result of this paper.
\subsection{The dynamic pressure with underlying currents}
\large
\begin{theorem}\label{the1}
In presence of a uniform underlying current with strength $k\neq 0$, the dynamic pressure in an irrotational regular wave trains

(1) attains its maximum value and minimum value at the wave crest and trough respectively for $k\neq c$,

(2) is identically zero for $k=c$.
\end{theorem}
\begin{proof} We analyze the case $k>c$, $k=c$ and $k<c$ respectively.

\vspace{3mm}

\noindent\emph{Case 1. $k>c$. }

\vspace{3mm}

Denote $k-c=\bar{c}>0$, then we have
\begin{equation}\label{3.5}
\bar{c}=\frac 1 L \int_0^L\psi_y(x,-d)dx>0.
\end{equation}
We claim that
\begin{equation}\label{3.6}
u(x,y)>c, \quad\text{for} \;\; -d\leq y \leq \eta(x).
\end{equation}
By the maximum principle, the $L$-periodicity in the $x$ variable of the harmonic function $\psi(x,y)$ ensures that its maximum and minimum in the domain
\begin{equation}\label{3.7}
D=\{(x,y):-d\leq y\leq \eta(x)\},
\end{equation}
will be attained on the boundary. This combined with the assumption $\bar{c}>0$ enforce that
\begin{equation}\label{3.8}
m<0.
\end{equation}
Since,

(a) if $m=0$, then $\psi$ is definitely constant throughout the domain $D$, resulting in $\bar{c}\equiv 0$ by its definition.

(b) if $m>0$, then the maximum value of $\psi$ is attained on the $y=-d$ and the Hopf's maximum principle yields $\psi_y(x,-d)<0$, which leads to $\bar{c}<0$.

This make a contradiction. For $m<0$, we get that $\psi$ attains is maximum and minimum
on $y=\eta(x)$ and $y=-d$ respectively. And a further application of Hopf's maximum principle leads to $\psi_y(x,0)>0$ and $\psi_y(x,-d)>0$. Thus $\psi_y>0$ on the boundary of $D$ and the maximum principle then ensures $\psi_y>0$ throughout $D$ since $\psi_y$ is harmonic function, then the claim \eqref{3.6} follows from the definition \eqref{2.8} of $\psi_y$.

Due to the anti-symmetry and the $L$-periodicity in the $x$ variableof $v$ and \eqref{2.5}, we have that $v=0$ along the lower and lateral boundaries of the domain
\begin{equation}\label{3.9}
D_+=\{(x,y):0<x<L/2, -d< y< \eta(x)\}.
\end{equation}
Along the upper boundary we have $\eta'(x)<0$, so the equation \eqref{2.6} and the claim \eqref{3.6} ensure that $v(x,\eta(x))\leq 0$ for $x\in[0,L/2]$. Since $v\leq 0$ on the whole boundary of $D_+$, we have from maximum principle that $v<0$ throughout $D_+$. Then the Hopf's maximum principle and the equations \eqref{2.3}-\eqref{2.4} yield
\begin{align}
&u_x(x,-d)=-v_y(x,-d)>0,\quad  \text{for} \quad x\in(0,L/2),\label{3.10}\\
&u_y(0,y)=v_x(0,y)<0,\quad  \text{for} \quad y\in(-d,\eta(0)),\label{3.11}\\
&u_y(L/2,y)=v_x(L/2,y)>0,\quad  \text{for} \quad y\in(-d,\eta(L/2)). \label{3.12}
\end{align}
Combined with \eqref{3.6} and \eqref{2.17}
\begin{equation}\label{3.13}
p=\rho g (Q-d)-\rho\frac {\psi^2_x+\psi^2_y} {2}=\rho g (Q-d)-\rho\frac {v^2+(u-c)^2} {2},
\end{equation}
\eqref{3.10}-\eqref{3.12} ensure that $p$ is strictly decreasing as we descend vertically in the fluid below the crest, it continues to decrease along the portion $\{(x,-d):0<x<L/2\}$ of the flat bed as $x$ decrease, and this strict decrease persists as we ascend vertically from the bed towards the surface, below the trough. Furthermore, in view of \eqref{2.7} and \eqref{2.17} and the monotonicity of the wave profile between the crest and a successive trough, we have that $p$ is also strictly decreasing as we descend from the crest towards the trough along the upper boundary of $D_+$.

On the other hand, by a direct calculation we have
\begin{equation}\label{3.14}
p_{xx}+p_{yy}=-\frac{2(p_x^2+p_y^2)} {\rho (\psi_x^2+\psi_y^2)}=\alpha p_x+\beta p_y,\quad \text{in}\;D,
\end{equation}
with
\begin{equation}\label{3.15}
\alpha=-\frac {2p_x} {\rho (\psi_x^2+\psi_y^2)},\quad \beta=-\frac {2p_y} {\rho (\psi_x^2+\psi_y^2)}
\end{equation}
where we have $\psi_x^2+\psi_y^2>0$ by \eqref{3.6}. Then the maximum principle ensures that the extrema of $p$ can only be attained on the boundary of $D_+$. Noting that the previous discussion about the behaviour of $p$ along the boundary of $D_+$, we conclude the result for $k>c$.

\vspace{3mm}

\noindent\emph{Case 2. $k=c$. }

\vspace{3mm}
For this case, we have
\begin{equation}\label{3.18}
\bar{c}=k-c=\frac 1 L \int_0^L\psi_y(x,-d)dx=0.
\end{equation}
Now we claim
\begin{equation}\label{3.19}
\psi(x,y)\equiv 0\quad \text{in}\quad D.
\end{equation}
Indeed, we know that the  harmonic function $\psi(x,y)$ attains its maximum and minimum on the boundary. And if $m\neq 0$, the Hopf's maximum principle yields $\psi_y(x,-d)\neq 0$. On the other hand, by the mean-value theorem, \eqref{3.18} implies that there must exist a point $(x_0,-d)$ such that $\psi_y(x_0,-d)=0$, which makes a contradiction. Thus we obtain $m=0$, which combined with the equations \eqref{2.11}-\eqref{2.12} and the maximum principle enforces \eqref{3.19}. Hence we have
\begin{equation}\label{3.21}
\psi_y=0\quad \text{and}\quad \psi_x=0,
\end{equation}
which in turn lead to
\begin{equation}\label{3.22}
u=c\quad \text{and} \quad v=0.
\end{equation}
Thus we get from \eqref{2.1} and \eqref{2.2} that
\begin{equation}\label{3.23}
P_x=0\quad \text{and} \quad P_y=-\rho g.
\end{equation}
Then differentiate \eqref{2.7} with respect to $x$ to get
\begin{equation}\label{3.24}
P_x(x,\eta(x))+P_y(x,\eta(x))\eta'(x)=0,
\end{equation}
which combined with \eqref{3.23} lead to
\begin{equation}\label{3.25}
\eta(x)=\eta_0.
\end{equation}
Performing a shift, we may assume the flat surface $\eta_0=0$. Hence \eqref{3.25}, \eqref{3.21} and \eqref{2.14} imply $Q=d$ and the total pressure thus reduced to the hydrostatic pressure
\begin{equation}\label{3.26}
P=P_{atm}-\rho g y
\end{equation}
by \eqref{2.16} and thus the dynamic pressure is identically zero.

\noindent\emph{Case 3. $k<c$. }

\vspace{3mm}
Denote $k-c=\bar{c}<0$, then we have
\begin{equation}\label{3.16}
\bar{c}=\frac 1 L \int_0^L\psi_y(x,-d)dx<0.
\end{equation}
A similar discussion as the case $k>c$ leads to
\begin{equation}\label{3.17}
u(x,y)<c, \quad\text{for} \;\; -d\leq y \leq \eta(x).
\end{equation}
Hence this case can be treated in a similar way as the case with the irrotational regular wave train without any current \cite{Co1} and we get that the dynamic pressure attains its maximum value at the wave crest and its minimum value at the wave trough.

\vspace{3mm}

Concluding the Case 1-3, we complete the proof of Theorem \ref{the1}.
\end{proof}
\subsection{The dynamic pressure in deep water}
The governing equations for irrotational regular wave trains in deep water are the following \cite{D3, Jo1}.
\begin{align}
(u-c)u_x+vu_y&=-\frac 1 \rho P_x,\quad\;\;\quad -d\leq y\leq \eta(x),\label{4.1}\\
(u-c)v_x+vv_y&=-\frac 1 \rho P_y-g,\!\quad -d\leq y\leq \eta(x),\label{4.2}\\
u_x+v_y&=0,\qquad\;\;\qquad\;-d\leq y\leq \eta(x),\label{4.3}\\
u_y&=v_x,\qquad \;\; \qquad -d\leq y\leq \eta(x),\label{4.4}\\
v&=(u-c)\eta_x,\;\;\:\quad \textrm{on} \quad y=\eta(x),\label{4.5}\\
P&=P_{atm},\;\;\qquad\quad \textrm{on} \quad y=\eta(x),\label{4.6}\\
(u,v)&\rightarrow (0,0) \;\;\qquad\quad \text{as}\;y\rightarrow -\infty \;\; \text{uniformly for}\; x\in\mr. \label{4.7}
\end{align}
Using the stream function $\psi$, we can reformulate the governing equations in the moving frame as follows.
\begin{align}
&\psi_{xx}+\psi_{yy}=0,\quad\;\;\quad \text{for}\; -d\leq y\leq \eta(x),\label{4.8}\\
&\psi=0,\quad\;\;\quad  \text{on}\; y=\eta(x),\label{4.9}\\
&\frac {\psi^2_x+\psi^2_y} {2g}+y=E\quad  \text{on}\; y=\eta(x),\label{4.10}\\
&\triangledown\psi\rightarrow (0,-c),\quad\;\text{as}\;y\rightarrow -\infty \;\; \text{uniformly for}\; x\in\mr,\label{4.11}
\end{align}
where the constant $E$ is the total head and it is expressed by
\begin{equation}\label{4.12}
\frac {\psi^2_x+\psi^2_y} {2g}+y+\frac {P-P_{atm}} {\rho g}=E
\end{equation}
throughout $\overline{D_\eta}=\{(x,y): -\infty <y\leq \eta(x)\}$. The left side of \eqref{4.12} is the Bernoulli's law. Then the total pressure can be recovered from
\begin{equation}\label{4.13}
P=P_{atm}+\rho g E-\rho g y-\rho\frac {\psi^2_x+\psi^2_y} {2}.
\end{equation}
Hence the dynamic pressure is given by
\begin{equation}\label{4.14}
p=P(x,y)-(P_{atm}-\rho g y)=\rho g E-\rho\frac {\psi^2_x+\psi^2_y} {2}.
\end{equation}
And we assume that there is no underlying current, that is \cite{D3}
\begin{equation}\label{k}
k=\frac 1 L \int_0^Lu(x,y_0)dx=0,
\end{equation}
where $y_0$ is any fixed depth below the wave trough level, which implies that
\begin{equation}\label{u}
u<c \quad \text{throughout}\;\;\overline{D_\eta}.
\end{equation}
Indeed, let $\bar{y}\in \mr^-$ be negative enough such that $\psi_y<0$ uniformly in $x$ for $-\infty<y\leq \bar{y}$: it follows from the equation \eqref{4.11} that such a $\bar{y}$ exists. On the other hand, by the maximum principle, the $L$-periodicity in the $x$ variable of the harmonic function $\psi(x,y)$ ensures that its maximum and minimum in the domain
$D_{\bar{y}}=\{(x,y):\bar{y}\leq y\leq \eta(x)\}$ will be attained on the boundary. Moreover, the Hopf's maximum principle and the fact $\psi_y(x,\bar{y})<0$ mean that the maximum value is attained on $y=\bar{y}$ and the minimum value must therefore be attained all along the free surface $y=\eta(x)$. Hence $\psi_y<0$ on the boundary of $D_{\bar{y}}$ and apply the maximum principle on the harmonic function $\psi_y$ yields $\psi_y<0$ throughout $D_{\bar{y}}$. Hence, $\psi_y<0$ throughout $\overline{D_\eta}$, which  then ensures \eqref{u}.

The following lemma is some results about the deep-water Stokes wave, which is crucial in our investigation of dynamic pressure in deep water.
\begin{lemma}\label{lem1}\cite{D3,D1}
Denote $D_\eta^+=\{(x,y):x\in(0,L/2),y\in(-\infty,\eta(x))\}$,
then the following strict inequalities hold
\begin{equation}\label{4.15}
v(x,y)=-\psi_x(x,y)> 0\;\; \text{and}\;\; P_x<0,\quad \text{for}\;\; (x,y)\in D_\eta^+,
\end{equation}
and on the crest and trough lines there have $v=0$ and $P_x=0$.
\end{lemma}

Now, it is the position to present the result about the dynamic pressure in deep water.
\begin{theorem}\label{the2}
The dynamic pressure in an irrotational regular wave train in deep water flows without underlying currents attains its maximum value at the crest and minimum value at the wave trough.
\end{theorem}
\begin{proof}
From Lemma \ref{lem1}, we know that $v>0$ throughout $D_\eta^+$. Then the Hopf's maximum principle of the harmonic function $v$ and the equations \eqref{4.3}-\eqref{4.4} yield
\begin{align}
&u_y(0,y)=v_x(0,y)>0,\quad  \text{for} \quad y\in(-\infty,\eta(0)),\label{4.16}\\
&u_y(L/2,y)=v_x(L/2,y)<0,\quad  \text{for} \quad y\in(-\infty,\eta(L/2)). \label{4.17}
\end{align}
Combined with \eqref{4.14} and \eqref{u}, \eqref{4.16}-\eqref{4.17} ensure that
\begin{equation}\label{4.18}
p=P(x,y)-(P_{atm}-\rho g y)=\rho g E-\rho\frac {v^2+(c-u)^2} {2}
\end{equation}
is strictly decreasing as we move down along the vertical half-line $[(0,\eta(0)), (0,-\infty)]$, and it is  strictly decreasing along  $[(L/2,\eta(L/2)), (L/2,-\infty)]$  as we move upwards towards the surface, below the trough. Furthermore, in view of \eqref{4.6} and \eqref{4.14} and the monotonicity of the wave profile between the crest and a successive trough, we have that $p$ is also strictly decreasing as we descend from the crest towards the trough along the upper boundary of $D_\eta^+$. On the other hand \eqref{4.16}-\eqref{4.17} and \eqref{4.7} tell us that $u$ decreases strictly towards zero along $x=0$ as y tends to minus infinity, whereas it increases strictly towards zero as $y \rightarrow -\infty$ along $x=L/2$, implying $u(0,y)>0$ for $y\in(-\infty, \eta(0)]$ and $u(L/2,y)<0$ for $y\in(-\infty, \eta(L/2)]$. Thus the maximum of $p$ along the boundary of $D_\eta^+$ can only be attained at $(0,\eta(0))$ and the minimum of $p$ along the boundary of $D_\eta^+$ can only be attained at $(L/2,\eta(L/2))$.

Now we prove the extrema of $p$ can not occur in the interior of the domain $D_\eta^+$. Assume that the contrary that there is a point $(x_0,y_0)\in D_\eta^+$ such that $p(x_0,y_0)$ is the extrema, then we have
\begin{align}
p_x&=-\rho \left(vv_x+(u-c)u_x\right)=0,\label{4.19}\\
p_y&=-\rho \left(vv_y+(u-c)u_y\right)=0,\label{4.20}
\end{align}
at $(x_0,y_0)$. Using the equations \eqref{4.3}-\eqref{4.4} and \eqref{u}, we can solve the above equations to get
\begin{equation}\label{4.21}
u_x(x_0,y_0)=u_y(x_0,y_0)=0.
\end{equation}
Hence
\begin{equation}\label{4.22}
P_x(x_0,y_0)=0
\end{equation}
follows by \eqref{4.1}, which contradicts to \eqref{4.15} in Lemma \ref{lem1}.
This combined with the previously discussions of the behaviour of $p$ along the boundary of $D_\eta^+$ complete the proof.
\end{proof}

\begin{remark} From the results of Theorem \ref{the1} and Theorem \ref{the2}, we know that the uniform underlying currents and infinite depth make no difference on the position of the extrema of the dynamic pressure.
\end{remark}

\vspace{0.5cm}
\noindent {\bf Acknowledgements.}
The work of Gao is partially supported  by the NSFC grant No. 11531006, National Basic Research Program of China (973 Program) No. 2013CB834100, PAPD of Jiangsu Higher Education Institutions, and the Jiangsu Collaborative Innovation Center for Climate Change.

\end{document}